\documentclass[a4paper,10pt]{article}

\usepackage[english]{babel}
\input xy
\xyoption{all}
\usepackage{amssymb,amsmath,amsthm,color}
\usepackage{eucal,url,amssymb,stmaryrd,enumerate,amscd}
\usepackage{amsmath}
\usepackage{amsfonts}

\newcommand{\norm}[1]{\left\Vert#1\right\Vert}

\newcommand{\R}{\mathbb{R}}
\newcommand{\lie}[1]{\mathfrak{#1}}     

\newcommand{\hook}{\lrcorner\,}

\newcommand{\SU}{\mathrm{SU}}
\newcommand{\SO}{\mathrm{SO}}

\newcommand{\dfn}[1]{\emph{#1}}

\theoremstyle{plain}
\newtheorem{proposition}{Proposition}[section]
\newtheorem{theorem}[proposition]{Theorem}
\newtheorem{lemma}[proposition]{Lemma}

\theoremstyle{definition}
\newtheorem{definition}[proposition]{Definition}

\theoremstyle{remark}

\usepackage[hypertex]{hyperref}      
\usepackage{amsfonts}
\usepackage{amsthm}

\newcommand{\Span}[1]{\operatorname{Span}\left\{#1\right\}}
\title{On seven dimensional quaternionic contact solvable Lie groups}
\author{Diego Conti, Marisa Fern\'andez and Jos\'e A. Santisteban}

\begin{document}

\maketitle

\begin{abstract}
We answer in the affirmative a question posed by Ivanov
and Vassilev \cite{IV1} on the existence of  a seven dimensional
quaternionic contact manifold
with closed fundamental $4$-form and non-vanishing
torsion endomorphism.
Moreover, we
show an approach to the classification of
seven dimensional solvable Lie groups having  an integrable left invariant
quaternionic contact
structure. In particular, we prove that the
unique seven dimensional nilpotent Lie group admitting such a structure
is the quaternionic Heisenberg group.
\end{abstract}

\vskip15pt{\small\textbf{2010 MSC classification}: 53C26; 53C25}\vskip5pt

\section{Introduction}

The notion of  quaternionic contact (qc) structure was introduced by Biquard in \cite{Biq1}, and it is
the natural geometrical structure that appears on the $(4n+3)$-sphere as the conformal 
infinity of the quaternionic hyperbolic space.
Such structures have been considered in connection with
the quaternionic contact Yamabe problem \cite{IMV1,IMV,IV2,Wei}.
Results about the CR-structure on the twistor space of a qc manifold
were given in \cite{Alt,Biq2,D,DIM}.

In general, a  qc structure on a differentiable manifold of dimension (4n+3)
is a distribution $H$ of codimension $3$ on $M$, called the horizontal space,
such that there exists a metric $g$ on $H$ and
a triplet $(\eta_1,\eta_2,\eta_3)$ of locally defined differential $1$-forms 
vanishing on $H$ and
such that the restrictions $d\eta_r|_H$ to $H$ of the $2$-forms $d\eta_r$ ($1\leq r\leq 3$)
are the local K\"ahler 2-forms
of an almost quaternion Hermitian structure on $H$.

The triplet of 1-forms $(\eta_1,\eta_2,\eta_3)$ is
determined up to a conformal factor and the action of $SO(3)$ on
$\mathbb{R}^3$. Therefore, $H$ is equipped with a conformal class
$[g]$ of Riemannian metrics and a rank $3$ bundle (the quaternionic bundle) $\mathbb{Q}$ of endomorphisms of $H$ such that $\mathbb{Q}$ is locally generated
by almost complex structures $(I_1, I_2, I_3)$ satisfying the quaternion relations.
The 2-sphere
bundle of one forms determines uniquely the associated metric and
a conformal change of the metric is equivalent to a conformal
change of the one forms.

Biquard in  \cite{Biq1} shows that
if $M$ is a qc manifold of dimension greater than $7$,
to every metric in the fixed conformal class $[g]$, one can associate
a  unique complementary distribution $V$  of $H$ 
in the
tangent bundle $TM$ such that $V$ is
locally generated by vector fields $\xi_1,\xi_2,\xi_3$
satisfying certain relations (see (\ref{bi1}) in section 2).
Using these vector fields $\xi_r$, we extend the metric $g$ on $H$
to a metric on $M$ by requiring
$\Span{\xi_1,\xi_2,\xi_3}=V\perp H$  and
$g(\xi_r,\xi_k)=\delta_{rk}$.
Moreover, in  \cite{Biq1} it is also proved that
$M$ has a canonical
 linear connection $\nabla$ preserving the qc structure and the splitting
$TM=H\oplus V$. This connection is
known as the Biquard connection.

However, if the dimension of $M$ is seven, there might be no vector fields
satisfying \eqref{bi1}.  Duchemin shows in \cite{D} that if
there are vector fields $\xi_r$ $(1\leq r\leq 3)$ satisfying the relations
\eqref{bi1} before mentioned, then the Biquard connection
is also defined on $M.$ In this case, the qc structure on the $7$-manifold
is said to be {\em integrable}.
In this paper, we assume the
integrability of the qc structure when we refer to a $7$-dimensional qc manifold.

If $M$ is a qc manifold with horizontal space $H,$
the restriction to $H$ of the  Ricci tensor of
$(g,\nabla)$ gives rise, on the one hand, to the qc-scalar curvature $S$ and,
on the other hand, to two
symmetric trace-free (0,2) tensor fields $T^0$ and $U$ defined on the distribution $H$
(see also Section $4$). The
tensors $T^0$ and $U$ determine the trace-free part of the Ricci tensor restricted to $H$ and can
also be expressed in terms of the torsion endomorphisms of the Biquard connection \cite{IMV}.
Moreover, the vanishing of the torsion endomorphisms of
the Biquard connection is equivalent to $T^0=U=0$ and if the dimension is at least eleven, then the
function $S$ has to be constant.
For any $7$-dimensional qc manifold, in \cite{Biq1,Biq2} it is proved that $U=0$ and
in \cite{IMV,IV1} it is shown that $S$ is constant if
the torsion endomorphism vanishes and the distribution $V$ is integrable (that is, $[V,V] \subset V$).

Associated to the $Sp(n)Sp(1)$
structure on the distribution $H$ of a qc structure,
one has the fundamental 4-form $\Omega$ defined (globally) on
$H$ by
\begin{equation}\label{fform}
\Omega=\omega_1\wedge\omega_1+\omega_2\wedge\omega_2+\omega_3\wedge\omega_3,
\end{equation}
where $\omega_r$, $1\leq r\leq 3$,  are the local K\"ahler $2$-forms
of the almost quaternion structure on $H$.
In \cite{AFISV} it is proved that for a seven dimensional manifold
with an integrable qc structure, the vertical space $V$
is integrable
if and only if the fundamental $4$-form $\Omega$
is closed.
Ivanov and Vassilev in \cite{IV1} prove
 that when the dimension of the manifold is greater than seven, the
4-form form $\Omega$ is closed if and only if the torsion endomorphism of the Biquard
connection vanishes. They raise the question
of the existence of a seven dimensional qc manifold with a closed
fundamental four form and a non-vanishing torsion endomorphism.
In this article, we answer the question in the affirmative
by proving the following.

\begin{theorem} \label{main-theorem}
There are seven dimensional manifolds with an integrable
qc structure
such that the fundamental four form $\Omega$ is closed but the torsion endomorphism
does not vanish.
\end{theorem}

Examples of qc manifolds can be found in \cite{Biq1,Biq2,IMV,D1}.
The compact homogeneous model is the sphere $S^{4n+3}$, considered as
the boundary at infinity of quaternionic projective $(n+1)$-space, while the
non-compact homogeneous model is the quaternionic Heisenberg group
$G({\mathbb{H}}) ={\mathbb{H}}^n \times {Im}(\mathbb{H})$ endowed with its
natural qc structure; in
fact, $G({\mathbb{H}})$ is isomorphic as a qc manifold to $S^{4n+3}$ minus a point,
via the quaternionic Cayley transform.
Moreover, an
extensively studied class of examples of quaternionic contact
structures are provided by the  3-Sasakian manifolds. We recall that a
$(4n+3)$-dimensional  Riemannian manifold $(M,g)$ is
called 3-Sasakian if the cone
metric $g_c=t^2g+dt^2$ on $C=M\times \mathbb{R}^+$ is a
hyperk\"ahler metric, namely, it has holonomy contained
in $Sp(n+1)$ \cite{BGN}. For any 3-Sasakian manifold, it was shown in \cite{IMV} that
the torsion endomorphism vanishes, and
 the converse is true if in addition the qc scalar curvature (see
\eqref{qscs}) is a positive constant.
Explicit examples of seven dimensional qc manifolds with zero or
non-zero torsion endomorphism were recently given in \cite{AFISV}.
Nevertheless, the fundamental $4$-form is non-closed on the
seven dimensional qc manifold with non-zero torsion endomorphism
presented in  \cite{AFISV}.

To prove Theorem \ref{main-theorem}, we consider
solvable $7$-dimensional Lie algebras $\lie{g}$ with a {\em normal ascending flag}, that is,
the dual space  $\lie{g}^*$ has a flag
$$
V^0\subsetneq \dotsb\subsetneq V^7=\lie{g}^*,
$$
$V^i$ being an $i$-dimensional subspace of $\lie{g}$ such that
$dV^i\subset \Lambda^2V^i$, where $d$ is the Chevalley-Eilenberg
differential on $\lie{g}^*$. Firstly, in Section $2$, we prove that if $\lie{g}$ has an integrable qc structure,
then $e^1, e^2, e^3$ and $e^4$ are in $V^6$ for any $\{e^1,\cdots,e^7\}$ basis of $\lie{g}^*$
{\em adapted} to the qc structure and for any normal ascending flag of $\lie{g}$. 
In Section $3$, we show that the unique
solvable $7$-dimensional Lie algebras having a normal ascending flag and an integrable qc structure
are the Lie algebra of the quaternionic Heisenberg group and the two Lie algebras defined by
\begin{align*}
de^1&=0,     \qquad  de^2={(1+\mu)} e^{12}- \mu e^{15}+ \mu e^{34}- \mu e^{46},\\
de^3&=- {(1+\mu)} e^{13}- {(2+3 \mu)} e^{24}- \mu e^{16}+ \mu e^{45}, \quad de^4=2  \mu e^{14},\\
de^5&=e^{12}+e^{34} - e^{46},\\
de^6&=e^{13}+e^{42} +  e^{45},\\
de^7&=e^{14}+e^{23} + \mu e^{56},\
\end{align*}
where $\mu=-1,-\frac{1}{3}$. For convenience in the notation, in Section $4$, we 
shall
denote by 
$\lie{g}_1$ and 
$\lie{g}_2$
the Lie algebras defined by $\mu=-1$ and $\mu=-\frac{1}{3}$, respectively.
We consider the simply connected solvable
Lie group $G_s$ with Lie algebra $\lie{g}_s$ 
$(s=1,2)$, and define an integrable  
qc structure on 
each $G_s$. We prove (see Theorems $4.2$ and $4.3$) that the fundamental $4$-form on 
$G_s$  is closed but the torsion endomorphism of the Biquard connection is non-zero. Finally, we notice that
the qc structure on $G_s$
 is not locally qc conformal to the standard flat qc
structure on the quaternionic Heisenberg group $G({\mathbb{H}})$.

\smallskip
\noindent {\bf Acknowledgments.} This work has been partially supported through Project MICINN (Spain) MTM2008-06540-C02-01, and the UFI (UPV/EHU) 11/52.
\section{Solvable 7-dimensional Lie algebras with a normal ascending flag}
In this section, we consider solvable 7-dimensional Lie algebras with a {\em normal ascending flag}
and an {\em integrable} qc structure. For such a Lie algebra $\lie{g}$, we study
the behaviour of a coframe $\{e^1,\cdots,e^7\}$ {\em adapted} to the qc structure.
First, we need some definitions and results about qc manifolds and the Biquard connection.

Let $(M, g, \mathbb{Q})$ be a {\em qc manifold} of dimension $4n+3$, that is,
$M$ has a horizontal distribution $H$ of dimension $4n$ with a metric $g$ satisfying:
i) $H$ is locally determined by the kernel of three differential 1-forms $\eta_r$ $(1\leq r \leq 3)$
on $M$;
ii) $H$ has an $Sp(n)Sp(1)$ structure, that is, it is
equipped with a rank-three bundle
$\mathbb Q$ consisting of
endomorphisms of $H$ locally generated
by three almost complex structures $I_1,I_2,I_3$ on $H$ satisfying
the identities of the imaginary unit quaternions,
$I_1I_2=-I_2I_1=I_3, \quad I_1I_2I_3=-id_{|_H}$, which are
Hermitian with respect to the metric $g$, i.e. $g(I_r.,I_r.)=g(.,.)$; and
iii) the following
compatibility conditions hold:
$ 2g(I_rX,Y) = d\eta_r(X,Y)$, for $1\leq r\leq3$ and for any $X,Y\in H$.

Biquard in \cite{Biq1} proves that if $M$ is a qc manifold of dimension $(4n+3)>7$,
there exists a canonical connection on $M$. In the following theorem, we recall
the properties that distinguish that connection.
\begin{theorem}\label{biqcon}\cite{Biq1}
Let $(M, g,\mathbb{Q})$ be a qc
manifold of dimension $4n+3>7$.
Then there exists a unique connection
$\nabla$ with
torsion $T$ on $M^{4n+3}$ and a unique supplementary subspace $V$ to $H$ in
$TM$, such that:
\begin{enumerate}
\item[i)]
$\nabla$ preserves the decomposition $H\oplus V$ and the $Sp(n)Sp(1)$ structure on $H$,
i.e. $\nabla g=0,  \nabla\sigma \in\Gamma(\mathbb Q)$ for a section
$\sigma\in\Gamma(\mathbb Q)$;
\item[ii)] the torsion $T$ on $H$ is given by $T(X,Y)=-[X,Y]_{|V}$;
\item[iii)] for $\xi\in V$, the endomorphism $T(\xi,.)_{|H}$ of $H$ lies in
$(sp(n)\oplus sp(1))^{\bot}\subset gl(4n)$;
\item[iv)] the connection on $V$ is induced by the natural identification $\varphi$ of
$V$ with the subspace $sp(1)$ of the endomorphisms of $H$, i.e.
$\nabla\varphi=0$.
\end{enumerate}
\end{theorem}
In the part iii), the inner product $\langle,\rangle$ of $End(H)$ is given by $\langle A,B\rangle = {
\sum_{i=1}^{4n} g(A(e_i),B(e_i)),}$ for $A, B \in End(H)$.

We shall call the above connection {\em the Biquard connection}.
Biquard \cite{Biq1} also described the supplementary subspace $V$, namely,
$V$ is (locally) generated by vector fields $\{\xi_1,\xi_2,\xi_3\}$,
such that
\begin{equation}  \label{bi1}
\begin{aligned} \eta_s(\xi_k)=\delta_{sk}, \qquad (\xi_s\lrcorner
d\eta_s)_{|H}=0,\\ (\xi_s\lrcorner d\eta_k)_{|H}=-(\xi_k\lrcorner
d\eta_s)_{|H}, \end{aligned}
\end{equation}
where $\lrcorner$ denotes the interior multiplication.

If the dimension of $M$ is seven, there might be no vector fields
satisfying \eqref{bi1}. Duchemin shows in \cite{D} that if we
assume, in addition, the existence of vector fields $\{\xi_1,\xi_2,\xi_3\}$ as in
\eqref{bi1}, then an analogue of Theorem~\ref{biqcon}
holds. In this case, the qc structure on the $7$-manifold is called {\em integrable}.

From now on, given a $7$-dimensional Lie algebra $\lie g$ whose dual
space is spanned by $\{ e^1,\ldots ,e^7\}$, we will write $e^{ij}= e^i\wedge e^j$,
$e^{ijk}= e^i\wedge e^j\wedge e^k$, and so forth. Moreover,
let us fix some language.
On a Lie algebra $\lie{g}$, an {\em integrable qc structure} can be characterized by the existence of a coframe $e^1,\dotsc, e^7$ with
\begin{equation} \label{eqn:qc}
\begin{aligned}
de^5&=e^{12}+e^{34} + f_2\wedge e^7- f_3\wedge e^6 \mod \Span{e^{56},e^{57},e^{67}},\\
de^6&=e^{13}+e^{42} +  f_3\wedge e^5- f_1\wedge e^7\mod \Span{e^{56},e^{57},e^{67}},\\
de^7&=e^{14}+e^{23} +  f_1\wedge e^6- f_2\wedge e^5\mod \Span{e^{56},e^{57},e^{67}},
\end{aligned}
\end{equation}
where the $ f_i$ are in $\Span{e^1,e^2,e^3,e^4}$. This condition is invariant under the action of $\R^*\times\SO(4)$, where the relevant representation of $\SO(4)$ is
\[\R^7=\R^4\oplus\Lambda^2_+(\R^4),\]
and $\lambda\in\R^*$ acts as \[\operatorname{diag}(\lambda,\lambda,\lambda,\lambda,\lambda^2,\lambda^2,\lambda^2).\]
\begin{definition}
Let $\lie{g}$ be a $7$-dimensional Lie algebra with an integrable qc structure, and
let $\{e^1,\dotsc, e^7\}$ be a basis of the dual space $\lie{g}^*$. We say that
$\{e^1,\dotsc, e^7\}$ is an \em
 adapted coframe to the qc structure on $\lie{g}$
if with respect to that basis, $\lie{g}$ is defined by equations as \eqref{eqn:qc}.
\end{definition}

Furthermore, we consider solvable $7$-dimensional Lie algebras $\lie{g}$ that admit a \dfn{normal ascending flag}. This means that there is a flag
\[\lie{g}_0\subsetneq \dotsb\subsetneq \lie{g}_7=\lie{g},\]
where $\lie{g}_k$ is a $k$-dimensional ideal of $\lie{g}$. In particular this implies that $\lie{g}$ is solvable  since $[\lie{g}_i,\lie{g}_i]\subset \lie{g}_{i-1}$; moreover, taking annihilators, we obtain a dual flag
\begin{equation}
\label{eqn:normalflag}
V^0\subsetneq \dotsb\subsetneq V^7=\lie{g}^*, \quad dV^i\subset \Lambda^2V^i.
\end{equation}
The following straightforward result will be useful in the sequel.
\begin{lemma}
\label{lemma:solvable}
Let $\lie{g}$ be a Lie algebra with a normal ascending flag, and let $\alpha$ be an element of $\lie{g}^*$.
\begin{itemize}
\item If $(d\alpha)^k\neq0$, then $\alpha\notin V^i$, $i<2k$. If in addition
\[\alpha\in V^{2k}, \quad (d\alpha)^k=\eta^1\wedge\dotsb\wedge\eta^{2k},\]
then $V^{2k}=\Span{\eta^1,\dotsb,\eta^{2k}}$.
\item If $\alpha\wedge(d\alpha)^k\neq0$, then $\alpha\notin V^i$, $i<2k+1$. If in addition
\[\alpha\in V^{2k+1},\quad \alpha\wedge(d\alpha)^k=\eta^1\wedge\dotsb\wedge\eta^{2k+1},\]
then $V^{2k}=\Span{\eta^1,\dotsb,\eta^{2k+1}}$.
\end{itemize}
\end{lemma}

\begin{lemma}
\label{lemma:structure}
Let $\lie{g}$ be a $7$-dimensional Lie algebra with a normal ascending flag and an integrable qc structure.
Fix an adapted coframe $\{ e^1,\ldots ,e^7\}$ to the qc structure and a flag as in \eqref{eqn:normalflag}. Then,
\begin{align*}
\dim\Span{e^5,e^6,e^7}\cap V^4&=0,&\\
   \dim\Span{e^5,e^6,e^7}\cap V^5&=1, \\
   \dim\Span{e^5,e^6,e^7}\cap V^6&=2.
\end{align*}
\end{lemma}
\begin{proof}

Observe that $\SU(2)_+\subset\SO(4)$ acts on $\Span{e^5,e^6,e^7}$ as $\SO(3)$ acts on $\R^3$.
Hence, if we had a nonzero element of $\Span{e^5,e^6,e^7}\cap V^4$, we could assume it is $e^5$. Since $e^5\wedge(de^5)^2$ is non-zero, this contradicts Lemma~\ref{lemma:solvable}.

It follows that  \[\dim\Span{e^5,e^6,e^7}\cap V^5\leq1,\]
because the intersection with $V^4$ is trivial, and $V^4$ has codimension one in $V^5$. On the other hand, equality must hold, because $V^5$ has codimension two in $V^7$.

The last equality is proved in the same way.
\end{proof}

\begin{lemma}
\label{lemma:e4}
Let $\lie{g}$ be a $7$-dimensional Lie algebra with a normal ascending flag and an integrable qc structure.
Fix an adapted coframe $e^1,\dotsc, e^7$ and a flag as in \eqref{eqn:normalflag}. Then $e^1$, $e^2$, $e^3$ and $e^4$ are in $V^6$.
\end{lemma}
\begin{proof}
By Lemma~\ref{lemma:structure}, we can act by an element of $\SO(4)$ and obtain that $e^5$ is in $V^5$ and $e^6$ is in $V^6$. Moreover by dimension count
\[\dim\Span{e^1,e^2,e^3,e^4}\cap V^6\geq 3.\]
Thus, up to $\SO(4)$ action we can assume that $e^1,e^2,e^3,e^5,e^6$ lie in $V^6$.
Hence,  by Lemma~\ref{lemma:structure}, $e^7$ is not in $V^6$.

We must show that $e^4$ is also in $V^6$. Suppose otherwise. Then there is some $a\in\R$ with $e^4+ae^7$ in $V^6$. Up to $\R^*$ action we can assume $a=1$, that is
\[V^6=\Span{e^1,e^2,e^3,e^4+e^7,e^5,e^6}.\]
Therefore, \eqref{eqn:qc} gives
\[de^{56}=(e^{36}-f_2\wedge e^6+e^{25}-f_1\wedge e^5)\wedge e^7\mod\Lambda^3V^6+\Span{e^{567}};\]
and, on the other hand, $e^{56}$ is in $\Lambda^2V^6$, so $f_2=e^3$ and $f_1=e^2$, and
\begin{equation}
 \label{eqn:de7_4}
 de^7=e^{14}+e^{23}+e^{26}-e^{35} \mod \Span{e^{56},e^{57},e^{67}}.
\end{equation}

Write
\[f_3=f_3'+\lambda e^4, \quad f_3'\in \Span{e^1,e^2,e^3}.\]
Then
\[de^5=e^{12}+e^{3}(e^4+e^7) -f_3'\wedge e^6 - \lambda e^{46}\mod  \Span{e^{56},e^{57},e^{67}},\]
but by construction $de^5$ has to be in $\Lambda^2V^6$, so
\[de^5=e^{12}+e^{3}(e^4+e^7)-f_3'\wedge e^6+  \lambda e^6\wedge(e^4+e^7)  \mod \Span{e^{56}}.\]
Imposing $(de^5)^3=0$ we get
\[de^5=e^{12}+e^{3}(e^4+e^7)+  \lambda e^6\wedge(e^4+e^7) -f_3'\wedge e^6.\]
If \[f_3'=\mu_1e^1+\mu_2 e^2+\mu_3 e^3,\]
then $(de^5)^2\wedge e^5$ should be equal to $2E^{1234}\wedge e^5$,
where
\[E^1=e^1+\mu_2e^6,\quad E^2= e^2-\mu_1 e^6,\quad E^3=e^3+\lambda e^6,\quad E^4= e^4+e^7-\mu_3 e^6.\]
Indeed, \[de^5=E^{12}+E^{34}.\]
By Lemma~\ref{lemma:solvable}, it follows that
\[V^5=\Span{E^1,E^2,E^3,E^4, e^5}.\]
We can rewrite \eqref{eqn:de7_4} as
\[de^7=e^{14}+e^{23}+e^{26}-e^{35}+ae^{57}+be^{67}+ ce^{56}.\]
Then computing $d^2e^7$
mod $\Lambda^3V^6$ we find
\begin{equation}
 \label{eqn:de1}
0=d(-e^1+ae^5+be^6)=d(-E^1+ae^5+(\mu_2+b)e^6).
\end{equation}
We claim that \[\mu_2+b=0.\]
Otherwise,  \[de^6=\frac1{\mu_2+b}(dE^1-ade^5)\in\Lambda^2V^5,\]
and on the other hand
\begin{equation}
 \label{eqn:de6}
 de^6=e^{13}+e^{42} + f_3'\wedge e^5 +\lambda e^{45}- e^{27}\mod \Span{e^{56},e^{57},e^{67}};
\end{equation}
thus, $\alpha_3$ is zero. Now
\[d^2e^7 = d((ae^5+(\mu_2+b)e^6-E^1)e^7+ce^{56})\mod \Lambda^3V^5dd{}{,}\]
which is not zero because $e^{236}$ is not in $\Lambda^3V^5$.

We can therefore assume that $b=-\mu_3$; then
\begin{equation}
\label{eqn:dE1}
dE^1=ade^5=a(E^{12}+E^{34}).
\end{equation}
Imposing that $de^6$ is in $\Lambda^2 V^6$, \eqref{eqn:de6} becomes
\[de^6= e^{13} + f_3'\wedge e^5+\lambda(e^4+e^7)e^{5}- e^2(e^4+e^{7})+ xe^{56},\]
for some real $x$.
In order to simplify the notation, we shall think of $E^1,E^2,E^3,E^4$ as a coframe on a $4$-dimensional vector space, defining an $\SU(2)$-structure; in particular a scalar product is defined, allowing us to take interior products of forms, 
as well as three complex structures. Explicitly, we set
\[J_1= \cdot\hook (E^{12}+E^{34}), \quad J_2= \cdot\hook (E^{13}+E^{42}), \quad J_3= \cdot\hook (E^{14}+E^{23}).\]
We denote by $\beta$ the projection of $f_3$ on $\Span{E^1,E^2,E^3,E^4}$. Thus,
\[de^6=E^{13}+E^{42}+\beta\wedge e^5 + (J_3\beta + xe^5)e^6.\]

Now define two derivations $\delta,\gamma$ on $\Lambda\Span{E^1,E^2,E^3,E^4}$, of degrees one and zero respectively, by the rule
\[d\eta=\delta\eta+e^5\wedge \gamma\eta.\]
Then $d^2=0$ implies
\[\delta\gamma=\gamma\delta, \quad \delta^2\eta+(E^{12}+E^{34})\wedge\gamma \eta=0,\]
and so $\delta$ determines $\gamma$ via
\begin{equation}
 \label{eqn:gammadetermined}
 \gamma\eta=J_1 *\delta^2\eta,
\end{equation}
for any $1$-form $\eta$.

In particular, \eqref{eqn:dE1} gives
\[\delta E^1=a(E^{12}+E^{34}), \quad \gamma E^1=0,\]
which is consistent with \eqref{eqn:gammadetermined} because $d^2e^5=0$ implies
\[\delta (E^{12}+E^{34})=0, \quad \gamma(E^{12}+E^{34})=0.\]
Due to \eqref{eqn:gammadetermined}, the latter can be rewritten as
\begin{equation}
\label{eqn:stardeltasq}
0=\sum_{i=1}^4 E^i\wedge *\delta^2 E^i = \sum *(E_i\hook \delta^2 E^i).
\end{equation}

Taking $d^2e^6$ and separating the components we find
\begin{gather*}
\gamma(J_3\beta)=0,\quad
\delta(J_3\beta) + x(E^{12}+E^{34})=0,\\
\delta(E^{13}+E^{42})=\beta\wedge(E^{12}+E^{34})+J_3\beta\wedge(E^{13}+E^{42})=2\beta\wedge(E^{12}+E^{34}),\\
\gamma(E^{13}+E^{42})=-\delta\beta+x(E^{13}+E^{42})+J_3\beta\wedge\beta.
\end{gather*}
Similarly,
\[de^7=E^{14}+E^{23}-E^3\wedge e^5+(E^2- J_2\beta+(c-\lambda) e^5)\wedge e^6+(ae^5- E^1)\wedge e^7;\]
taking $d$ again and separating the components we get
\begin{gather*}
\delta(E^{14}+E^{23})=-3E^{123}-J_2\beta\wedge(E^{13}+E^{42}),\\
\gamma(E^{14}+E^{23})+\beta\wedge (E^2-J_2\beta)-(c-\lambda)(E^{13}+E^{42})+a(E^{14}+E^{23})=0,\\
\delta(E^2-J_2\beta)+(c-\lambda)(E^{12}+E^{34})+(E^2-J_2\beta)(-J_3\beta-E^1)=0,\\
\gamma(E^2-J_2\beta)+(c-\lambda)(-J_3\beta-E^1)-(a-x)(E^2-J_2\beta)=0.
\end{gather*}
In order to make use of these equations we need $E_2-J_2\beta$ to be nonzero, so let us assume first
\begin{equation}
\label{eqn:assumebetaE4}
\beta=E^4,
\end{equation}
and obtain a contradiction. Indeed, in this case  $c=\lambda$ and $a=-x=0$, for
\[0=(E^{12}+E^{34})\wedge \gamma (E^{12}+E^{34})=(E^{13}+E^{42})\wedge \gamma (E^{14}+E^{23}).\]
Hence, we have
\begin{gather*}
\delta E^1=0,\quad
\delta(E^{12}+E^{34})=0, \quad
\gamma(E^{12}+E^{34})=0,\\
\delta(E^{13}+E^{42})=2E^{124},\quad
\gamma(E^{13}+E^{42})=-\delta E^4-E^{14},\\
\delta(E^{14}+E^{23})=-2E^{123},\quad
\gamma(E^{14}+E^{23})=0.
\end{gather*}
In particular, since $\gamma$ is a derivation, we see that $\gamma(E^{13}+E^{42})$ gives zero on wedging with
\begin{equation}
 \label{eqn:omega_is}
E^{12}+E^{34}, \quad E^{13}+E^{42}, \quad E^{14}+E^{23},
\end{equation}
and so the same holds of $\delta E^4+E^{14}$. Thus $\delta E^4$ has the form
\begin{equation}
 \label{eqn:xyz}
\delta E^4=-\frac12(E^{14}+E^{23})+x(E^{12}-E^{34})+y(E^{13}-E^{42})+z(E^{14}-E^{23}).
\end{equation}
Wedging $E^1$ with the forms \eqref{eqn:omega_is} and applying $\delta$ we obtain
\[\delta E^{134}=0=\delta E^{123}=\delta E^{124}.\]
Similarly,
\[\delta E^{234}=\delta (E^{14}+E^{23})\wedge E^4+(E^{14}+E^{23})\wedge \delta E^4=-3E^{1234}.\]
Using \eqref{eqn:xyz} we get
\begin{multline}
\label{eqn:delta2E4}
\delta^2E^4\wedge E^4=E^{1234} +(x\delta (E^{12}-E^{34})+y\delta(E^{13}-E^{42})+z\delta(E^{14}-E^{23}))\wedge E^4\\
=E^{1234}+x\delta E^{124}-x(E^{12}-E^{34})\wedge\delta E^4+y\delta E^{134}-y(E^{13}-E^{42})\wedge\delta E^4-zdE^{234}-z(E^{14}-E^{23})\wedge\delta E^4\\
=E^{1234}(1+2x^2+2y^2+3z +2z^2).
\end{multline}
Observe that
\[\delta E^{14}=-E^1\wedge\delta E^4=E^1\wedge\gamma (E^{13}+E^{42})=-\gamma E^{124},\]
so
\[\delta^2E^4=-\delta E^{14}-2\gamma E^{124}=\delta E^{14}.\]
By Equation \eqref{eqn:xyz},
\[\delta E^{14}\wedge E^4=-E^{14}\wedge \delta E^4=(z+\frac12)E^{1234};\]
comparing with \eqref{eqn:delta2E4} we obtain
\[2x^2+2y^2+2z^2+2z+\frac12=0,\]
and therefore
\[\delta E^4=-E^{14}.\]
In particular, we see that $\gamma=0$ and $\delta$ defines a $4$-dimensional Lie algebra characterized by the equations
\[\delta E^1=0, \quad \delta E^4=-E^{14}, \quad \delta(E^{12}+E^{34})=0, \quad \delta(E^{13}+E^{42})=2E^{124},\quad\delta E^{23}=-2E^{123}.\]
It is easy to check that no such Lie algebra 
exists. Indeed
these linear conditions on $\delta$ imply
\begin{gather*}
\delta E^2=p(E^{12}-E^{34})+q(E^{13}-E^{42})-\frac32(E^{12}+E^{34}),\\
\delta E^3=q(E^{12}-E^{34})-p (E^{13}-E^{42})-\frac12(E^{13}+E^{42})
\end{gather*}
for some $p,q$; but then $\delta^2$ is not zero. Summing up, \eqref{eqn:assumebetaE4} leads to a contradiction.

\smallskip
We can therefore assume that
\[0\neq\tilde E^1=E^1+J_3\beta.\]
Then
\[\tilde E^1, \quad \tilde E^2=J_1\tilde E^1=E^2-J_2\beta, \quad \tilde E^3=J_2\tilde E^1=E^3+J_1\beta, \quad \tilde E^4=J_3\tilde E^1=E^4-\beta\]
is an orthonormal basis up to a scale factor, hence it also satisfies \eqref{eqn:stardeltasq}.
We compute
\begin{equation}
\label{eqn:delta2E2}
\delta^2\tilde E^2=-(c-\lambda)(E^{12}+E^{34})\tilde E^1  -(a-x)\tilde E^2(E^{12}+E^{34}),
\end{equation}
and therefore
\[\tilde E^2\hook \delta^2\tilde E^2 = -(a-x)\tilde E^{34}.\]
 Now \eqref{eqn:stardeltasq} implies that $a-x=0$ and
\[\tilde E^3\wedge *\delta^2\tilde E^3+\tilde E^4\wedge\delta^2\tilde E^4=0,\]
so in particular $*\delta^2\tilde E^3$ and $*\delta^2\tilde E^3$ are in the span of $\tilde E^3$ and $\tilde E^4$.

Using \eqref{eqn:gammadetermined}, we find
\[J_1\gamma(\tilde E^{14}+\tilde E^{23})=-\tilde E^1\wedge *\delta^2\tilde E^4-*\delta^2\tilde E^2\wedge \tilde E^3 - \tilde E^2\wedge*\delta^2\tilde E^3\]
On the other hand, we know that
\[J_1\gamma(E^{14}+E^{23})-J_1\beta\wedge\tilde E^1+(c-\lambda)(E^{13}+E^{42})-a(E^{14}+E^{23})=0.\]
Comparing the two expressions and using \eqref{eqn:delta2E2}, we find
\[- \tilde E^1\wedge *\delta^2\tilde E^4 - \tilde E^2\wedge*\delta^2\tilde E^3+(c-\lambda)\tilde E^{23}-\norm{\tilde E^1}^2J_1\beta\wedge\tilde E^1+(c-\lambda)(\tilde E^{13}+\tilde E^{42})-a(\tilde E^{14}+\tilde E^{23})=0 .\]
This shows that $J_1\beta$ has no component along $\tilde E^2$, so
\[\norm{\tilde E^1}^2J_1\beta = \beta\hook\tilde E^{34} \mod \tilde E^1;\]
using the fact that $*\delta^2\tilde E^3$ and $*\delta^2\tilde E^4$ are in the span of $\tilde E^3$ and $\tilde E^4$, we deduce
\[*\delta^2\tilde E^3=-(c-\lambda)\tilde E^4-(a-c+\lambda)\tilde E^3, \quad *\delta^2\tilde E^4= \beta\hook\tilde E^{34}+ (c-\lambda)\tilde E^3-a\tilde E^4,\]
and therefore by \eqref{eqn:stardeltasq}
\[*\delta^2\tilde E^4= -(c-\lambda)\tilde E^3\mod \tilde E^4;\]
it follows that
\begin{equation}
\label{eqn:sequals}
\beta\hook\tilde E^{34} =-2(c-\lambda)\tilde E^3\mod \tilde E^4.
\end{equation}
Similarly, we compute
\begin{multline*}
J_1\gamma(\tilde E^{13}+\tilde E^{42}) = - \tilde E^2\wedge *\delta^2\tilde E^3 +*\delta^2\tilde E^4\wedge \tilde E^1+\tilde E^3\wedge *\delta^2\tilde E^2\\
=-(c-\lambda)(\tilde E^{13}+\tilde E^{42})+a(\tilde E^{14}+\tilde E^{23})+\norm{\tilde E^1}^2J_1\beta\wedge \tilde E^1.
\end{multline*}
Comparing with
\[\gamma(E^{13}+E^{42})=-\delta\beta+a(E^{13}+E^{42})+J_3\beta\wedge\beta,\]
we obtain
\[(c-\lambda-a)( E^{13}+ E^{42})-a( E^{14}+ E^{23}) +\beta\wedge  (J_3\beta-\tilde E^2) +\delta\beta=0.\]
Taking $\delta$,
\[aE^{123}+(c-\lambda-3a)( E^{13}+ E^{42})\wedge (\tilde E^2)+a( E^{14}+ E^{23})\wedge (-\tilde E^2)-\beta \wedge \tilde E^{12} + \delta^2\beta=0,\]
and therefore taking $*$
\begin{equation}
 \label{eqn:starofwhich}
a\beta-(c-\lambda-4a)\tilde E^4-
a\tilde E^3-\beta\hook\tilde E^{34}+ *\delta^2\beta=0.
\end{equation}
So we have two cases.

\emph{i}) If $a=c-\lambda$,
write \[\beta=r\tilde E^3+s\tilde E^4 \mod \tilde E^2;\]
working mod $\tilde E^2$, \eqref{eqn:starofwhich} gives\
\begin{multline*}0=\tilde E^3(ar ) + \tilde E^4(as+2a-ar )
+ (s-1)(\beta\hook\tilde E^{34}+ a\tilde E^3-a\tilde E^4)\\
=\tilde E^3(ar-as+a ) + \tilde E^4(3a-ar+ (s-1)r\norm{\tilde E^3}^2),\end{multline*}
which implies
\[a=0=r=s,\]
since $s=\frac{2a}{\norm{\tilde E_1}^2}$ by \eqref{eqn:sequals}.
Then $\beta$ is a multiple of $\tilde E^2$, and $\gamma$, $\delta^2$ are zero. 
In particular,
\[\delta\beta=J_3\beta\wedge\beta\]
is linearly dependent on
\[\delta \tilde E^2=\tilde E^{12},\]
implying that $\beta=0$.
Summing up, $\delta$ defines a $4$-dimensional Lie algebra characterized by the equations
\[\delta E^1=0, \quad \delta E^2=E^{12}, \quad \delta E^{34}=0=\delta (E^{13}+E^{42}), \quad \delta (E^{14}+E^{23})=-3E^{123}.\]
Much like in the case that $\beta=E^4$, one verifies that no such Lie algebra exists,
for these linear  conditions on $\delta$ imply
\begin{gather*}
\delta E^3=p(E^{13}-E^{42})+q(E^{14}-E^{23})+\frac12(E^{13}+E^{42}),\\
\delta E^4=q(E^{13}-E^{42})-(p+5/2) (E^{14}-E^{23})+2(E^{14}+E^{23}),
\end{gather*}
where $p$ and $q$ are real numbers; but then $\delta^2$ is not zero.

\emph{ii}) Suppose $a\neq c-\lambda$; then \eqref{eqn:starofwhich}  implies that $\beta$ lies in the span of $\tilde E^3$ and $\tilde E^4$, so $J_3\beta$ lines in the span of $\tilde E^1$ and $\tilde E^2$. Thus $\tilde E^1$ and $J_3\beta$ are linearly dependent, for otherwise
\[\delta\tilde E^2\in \Span{E^{12}+E^{34}},\]
which is absurd. It follows that $J_3\beta$ is a multiple of $\tilde E^1$, say
$\beta = s\tilde E^4$.
Then
\[1=\norm{E^4}^2=\norm{\tilde E^4+\beta}^2=(1+s)^2\norm{\tilde E^4},\]
so
\[c-\lambda=\frac{s}{2(1+s)^2},\]
and
\[*\delta^2\tilde E^3=-\frac{s}{2(1+s)^2}\tilde E^4+\left(\frac{s}{2(1+s)^2}-a\right)\tilde E^3, \quad *\delta^2\tilde E^4= -\frac{s}{2(1+s)^2}\tilde E^3-a\tilde E^4.\]
Then \eqref{eqn:starofwhich} gives
\[-\left(\frac{s}{2(1+s)^2}-4a\right)\tilde E^4+\left(\frac{s^2}{2(1+s)^2}-a\right)\tilde E^3=0,\]
so both $c-\lambda$ and $a$ are zero, which is absurd.
\end{proof}

\section{Classification of $7$-dimensional qc Lie algebras with a normal ascending flag}
\label{classification}
In this section we carry out the classification of $7$-dimensional  Lie algebras with an integrable qc structure and a normal ascending flag.

\begin{proposition}
There are exactly three non-isomorphic Lie algebras 
 of dimension $7$,
with an integrable qc structure and a normal ascending flag, namely
\[(0,0,0,0,12+34,13+42,14+23)\]
and
\begin{equation}  \label{solvable-examples}
\begin{aligned}
de^1&=0\\
de^2&={(1+\mu)} e^{12}- \mu e^{15}+ \mu e^{34}- \mu e^{46}\\
de^3&=- {(1+\mu)} e^{13}- {(2+3 \mu)} e^{24}- \mu e^{16}+ \mu e^{45}\\
de^4&=2  \mu e^{14}\\
de^5&=e^{12}+e^{34} - e^{46}\\
de^6&=e^{13}+e^{42} +  e^{45}\\
de^7&=e^{14}+e^{23} + \mu e^{56}\\
\end{aligned}
\end{equation}
where $\mu=-1,-1/3$.
\end{proposition}
\begin{proof}
Let $\lie{g}$ be a  Lie algebra with an integrable qc structure and a fixed flag $V^i$ as in \eqref{eqn:normalflag}. By Lemma~\ref{lemma:e4}, we know that $e^1,e^2,e^3,e^4$ are in $V^6$; moreover, by the argument in the proof of that same lemma, we can assume that $e^5$ is in $V^5$ and $e^6$ is in $V^6$.

The characterization of the $V^i$ implies $\alpha_1=0=\alpha_2$, and
\begin{align*}
de^5&=e^{12}+e^{34} -\alpha_3\wedge e^6 \mod \Span{e^{56}},\\
de^6&=e^{13}+e^{42} + \alpha_3\wedge e^5\mod \Span{e^{56}},\\
de^7&=e^{14}+e^{23} \mod \Span{e^{56},e^{57},e^{67}}.
\end{align*}
We claim that
\begin{equation}
 \label{eqn:de7}
 de^7=e^{14}+e^{23} \mod \Span{e^{56}}.
\end{equation}
Indeed $de^{56}$ and $d(e^{14}+e^{23})$ are in $\Lambda^3V^6$, whereas
\[de^{57}=e^{127}+e^{347} -\alpha_3\wedge e^{67}\mod (\Lambda^3V^6 + \Span{e^{567}}),\]
\[de^{67}=e^{137}+e^{427} +\alpha_3\wedge e^{57}\mod (\Lambda^3V^6 + \Span{e^{567}}).\]
Thus, $d^2e^7=0$ implies \eqref{eqn:de7}.

Let us consider the splitting
 \[\Lambda^h\lie{g}^*=\bigoplus_{p+q=h} \Lambda^{p,q}=\bigoplus_{p+q=h}\Lambda^p\Span{e^1,e^2,e^3,e^4}\wedge\Lambda^qq\Span{e^5,e^6,e^7}.\]
Observe that
\[de^{56}=e^{126}+e^{346}-e^{135}-e^{425};\]
in particular $de^{56}$ has no (1,2)-component. Thus \eqref{eqn:de7} implies that $d(e^{14}+e^{23})$ has no (1,2)-component either, whence
\begin{equation}
\label{eqn:no02}
d\Lambda^{1,0}\subset\Lambda^{2,0}\oplus\Lambda^{1,1}.
\end{equation}
Then
\[0=(d^2e^5)^{1,2}=-d(\alpha_3\wedge e^6)^{1,2};\]
and the same holds of $\alpha_3\wedge e^5$.

We know that  $de^5$ equals $e^{12}+e^{34} -\alpha_3\wedge e^6$ plus a multiple $\epsilon$ of $e^{56}$, but if this multiple were nonzero, then
\[(de^5)^3=6\epsilon e^{123456}\neq 0,\]
contradicting Lemma~\ref{lemma:solvable} and the assumption $e^5\in V^5$. So, for some constants $\lambda,\mu\in\R$, we have
\begin{align*}
de^5&=e^{12}+e^{34} -\alpha_3\wedge e^6,\\
de^6&=e^{13}+e^{42} + \alpha_3\wedge e^5 +\lambda e^{56},\\
de^7&=e^{14}+e^{23} + \mu e^{56}.
\end{align*}
Then
\[0=d(\alpha_3\wedge e^6)^{1,2}=d\alpha_3^{1,1}\wedge e^6-\lambda \alpha_3\wedge e^{56}, \quad 0=d(\alpha_3\wedge e^5)^{1,2}=d\alpha_3^{1,1}\wedge e^5,\]
whence
\[(d\alpha_3)^{1,1}=\lambda \alpha_3\wedge e^5.\]
We have two cases, according to whether $\alpha_3$ is zero or not.

\emph{a}) If $\alpha_3=0$,
we know that $e^5$ is in $V^5$, and $e^5\wedge (de^5)^2=2e^{12345}$; hence,
\[V^5=\Span{e^1,e^2,e^3,e^4,e^5}.\]
Then
\begin{align*}
0&=d^2e^6=\lambda (e^{126}+e^{346})  \mod \Lambda^3V^5,\\
0&=d^2e^7=\mu (e^{126}+e^{346})  \mod \Lambda^3V^5.
\end{align*}
These equations imply that $\lambda=\mu=0$.

Now $\Span{e^1,\dotsc, e^4}$ intersects $V^4$ in a space of dimension at least three, so up to $\SO(4)$ action we can assume
\[V^4=\Span{e^1,e^2,e^3,e^4+ae^5}, \quad a\in\R,\]
whence
\[de^4=-ade^5 = -ae^{34}=a^2e^{35} \mod \Lambda^2V^4.\]
Therefore
\[0=d^2e^5=-a de^3\wedge e^5\mod \Lambda^3V^4,\]
\[0=d^2e^6=a^2 e^{235} + ae^5\wedge de^2  \mod \Lambda^3V^4,\]
\[0=d^2e^7=-a^2 e^{135} - ae^5\wedge de^1  \mod \Lambda^3V^4.\]
We claim that $a=0$. In fact if $a\neq0$, we see that $de^1$, $de^2$ and $de^3$ are in $\Lambda^2\Span{e^1,e^2,e^3}$, and so, taking $d^2$ of $e^5,e^6,e^7$, it follows that $de^4$ must also be in $\Lambda^2\Span{e^1,e^2,e^3,e^4}$. Denoting by $e_1,\dotsc, e_7$ the basis of $\lie{g}$ dual to $e^1,\dotsc, e^7$, we see that $\Span{e_5,e_6,e_7}$ is an ideal, and  $\lie{g}$ projects onto a hyperk\"ahler 4-dimensional, solvable algebra. This has to be abelian, because the corresponding Lie group is a homogeneous Ricci-flat manifold, hence flat by~\cite{Alek}. This implies that $e^4$ is closed, which contradicts the assumption $a\neq0$.
Consequently, $a=0$. This implies that $\Span{e_5,e_6,e_7}$ is an ideal, so by the same argument as above, $\Span{e_1,e_2,e_3,e_4}$ is abelian, and
\[\lie{g}=(0,0,0,0,12+34,13+42,14+23).\]

\emph{b}) Suppose $\alpha_3$ is non-zero. In this case, up to $\SO(4)$ action, we can assume that $\alpha_3$ is a multiple of $e^4$, and up to $\R^*$ action, we obtain $\alpha_3=e^4$. The equations become
\begin{align*}
de^5&=e^{12}+e^{34} - e^{46}\\
de^6&=e^{13}+e^{42} +  e^{45} +\lambda e^{56}\\
de^7&=e^{14}+e^{23} + \mu e^{56}\\
(de^4)^{1,1}&=\lambda e^{45}.
\end{align*}
In particular, looking at $(de^5)^2\wedge e^5$ we compute
\[V^5=\Span{e^1,e^2,e^3+e^6,e^4,e^5,}.\]
Thus
\[de^3+de^6=\gamma_3\wedge e^5+\beta_3\wedge (e^3+e^6) + h (e^{35}+e^{65})\mod \Lambda^2\Span{e^1,e^2,e^4}\]
where from now on the $\gamma_i$, $\beta_i$ are in $\Span{e^1,e^2,e^4}$. We can determine $h$ by
\[de^3 = -e^{13}-e^{45}-\lambda e^{56}+\gamma_3\wedge e^5+\beta_3\wedge (e^3+e^6)+ h (e^{35}+e^{65})\mod \Lambda^2\Span{e^1,e^2,e^4},\]
which by \eqref{eqn:no02} implies $h=-\lambda$.
 Similarly,
 \begin{align*}
 de^1&=\gamma_1\wedge e^5+\beta_1\wedge (e^3+e^6)\mod \Lambda^2\Span{e^1,e^2,e^4},\\
 de^2&=\gamma_2\wedge e^5+\beta_2\wedge (e^3+e^6)\mod \Lambda^2\Span{e^1,e^2,e^4}.
 \end{align*}
Now observe that
\[(d^2e^7)^{2,1}=(de^{14})^{2,1}+(de^2)^{1,1}\wedge e^3-e^2\wedge (de^{3})^{1,1}+\mu(e^{126}+e^{346}-e^{135}-e^{425});\]
therefore
\[0=(de^2)^{1,1}\wedge e^3 +\mu(e^{346}-e^{135})+\lambda e^{235} \mod \Lambda^3\Span{e^1,e^2,e^4,e^5,e^6},\]
so the $(1,1)$-component of $de^2$ is $-\mu e^{46}-\mu e^{15} +\lambda e^{25}$,
and
\[de^2=-\mu e^{15}+\lambda e^{25}-\mu e^4\wedge (e^3+e^6)\mod \Lambda^2\Span{e^1,e^2,e^4}.\]
Now
\begin{multline*}
0=(d^2e^6)^{2,1}=(de^{1})^{1,1}\wedge e^3-e^1\wedge (de^3)^{1,1}+\lambda e^{346}-\lambda e^{135} \mod \Lambda^3\Span{e^1,e^2,e^4,e^5,e^6}\\
=(de^{1})^{1,1}\wedge e^3+\lambda e^{346}\mod \Lambda^3\Span{e^1,e^2,e^4,e^5,e^6},
\end{multline*}
so
\[(de^{1})^{1,1} =-\lambda e^{46}.\]
On the other hand
\begin{multline*}
(d^2e^5)^{2,1}=(de^{12})^{2,1}+(de^3+de^6)^{1,1}\wedge e^4-\lambda e^{345}-e^6\wedge (de^4)^{2,0}\\
= -\lambda e^{125} - \gamma_3\wedge e^{45}   \mod \Lambda^3\Span{e^1,e^2,e^4,e^6},
\end{multline*}
which implies that $\lambda=0$ and $\gamma_3=ke^4$ for some $k\in\R$,
i.e.
\begin{align*}
de^1&=0 &&\mod \Lambda^2\Span{e^1,e^2,e^4},\\
de^2&=-\mu e^{15}-\mu e^4\wedge (e^3+e^6) &&\mod \Lambda^2\Span{e^1,e^2,e^4},\\
de^3+de^6&=ke^{45}+\beta_3\wedge (e^3+e^6) &&\mod \Lambda^2\Span{e^1,e^2,e^4},\\
de^4&=0 &&\mod \Lambda^2\Span{e^1,e^2,e^4},\\
de^5&=e^{12}+e^{34} - e^{46},\\
de^6&=e^{13}+e^{42} +  e^{45},\\
de^7&=e^{14}+e^{23} + \mu e^{56}.
\end{align*}
Assume first that $\mu=0$. Then
\begin{align*}d^2e^5&= -e^3\wedge de^4 + \beta_3\wedge e^{34} &&\mod \Lambda^3\Span{e^1,e^2,e^4,e^5,e^6},\\
d^2e^6&= de^1\wedge e^3-e^1\wedge \beta_3\wedge e^3 &&\mod \Lambda^3\Span{e^1,e^2,e^4,e^5,e^6},\\
d^2e^7&= de^2\wedge e^3-e^2\wedge \beta_3\wedge e^3- e^{123} &&\mod \Lambda^3\Span{e^1,e^2,e^4,e^5,e^6}.
\end{align*}
With an appropriate change in the definition of $k$, we obtain
\begin{align*}
de^1&=e^1\wedge \beta_3, &
de^2&=e^2\wedge\beta_3+e^{12},\\
de^3&=-e^{13}+ke^{45}+\beta_3\wedge (e^3+e^6) &&\text{mod } \Lambda^2\Span{e^1,e^2,e^4},\\
de^4&=-\beta_3\wedge e^{4},&
de^5&=e^{12}+e^{34} - e^{46},\\
de^6&=e^{13}+e^{42} +  e^{45},&
de^7&=e^{14}+e^{23}.
\end{align*}
Thus $d\beta_3=\beta_3(e_2)e^{12}$. Moreover,
\[0=d^2e^4=-d\beta_3\wedge e^4,\]
so $d\beta_3=0$. Thus
\[0=d^2e^2=e^{12}\wedge\beta_3 -2e^{12}\wedge \beta_3,\]
which implies $\beta_3$ is a multiple of $e^1$. More precisely,
\[0=d^2e^6=-ke^{145}-\beta_3\wedge e^{45}+2(\beta_3-e^1)e^{24}\]
implies $\beta_3=e^1$. But then
\[d^2e^7=-e^2\wedge de^3\neq0,\]
which is absurd.

\smallskip
Thus $\mu\neq0$.
Then $de^{12}$ and $de^{24}$ are linearly independent (mod $e^{124}$), and 
in consequence the
exact forms in  $\Lambda^2\Span{e^1,e^2,e^4}$ are multiples of $e^{14}$; in particular,
\[de^1,de^4\in\Span{e^{14}}.\]
By Lemma~\ref{lemma:solvable}, $V^4$ contains no linear combination of the form 
$e^5+ae^1$, $e^5+ae^4$, which means that $e^1,e^4$ are in $V^4$. So a linear combination of $e^1,e^4$ is in $V^3$.

Now
\[
(d^2e^6)^{2,1}=(1-k-\mu)e^{145} + e^{16}\wedge\beta_3 + de^4\wedge e^5
\]
shows that $\beta_3$ is a multiple of $e^1$ and
\[de^4=(k+\mu-1)e^{14}.\]
Similarly,
\begin{multline*}
(d^2e^5)^{2,1}=(de^{12})^{2,1}+(de^3+de^6)^{1,1}\wedge e^4-e^6\wedge de^4\\
= (\mu e^{14} -\beta_3\wedge e^4-(k+\mu-1)e^{14})\wedge e^{6}
\end{multline*}
shows that \[\beta_3=(1-k)e^1.\]
Finally, we have
\[ (d^2e^7)^{2,1}=(1-k)e^{245}+(1-k)e^{126}+\mu e^{126}+\mu e^{245}.\]
Consequently, $k=\mu+1$ and
\begin{align*}
de^1&=0 &&\mod e^{14},\\
de^2&=-\mu e^{15}-\mu e^4\wedge (e^3+e^6) &&\mod \Lambda^2\Span{e^1,e^2,e^4},\\
de^3+de^6&=(\mu+1)e^{45}-\mu e^1\wedge (e^3+e^6) &&\mod \Lambda^2\Span{e^1,e^2,e^4},\\
de^4&=2\mu e^{14},\\
de^5&=e^{12}+e^{34} - e^{46},\\
de^6&=e^{13}+e^{42} +  e^{45},\\
de^7&=e^{14}+e^{23} + \mu e^{56}.
\end{align*}
Imposing  $d^2=0$, a straightforward computation leads to \eqref{solvable-examples},
with $\mu=-1,-1/3$. Notice that the two resulting Lie algebras are non-isomorphic because 
their second cohomology groups $H^2(\lie{g^*})$ are different. In fact, $H^2(\lie{g^*})$ is $2$-dimensional if $\mu=-1$ but it is zero 
if $\mu=-1/3$.
One can check that a normal ascending flag exists in both cases by setting
\begin{gather*}
V^1=\Span{e^1}, \quad V^2=\Span{e^1,e^4}, \quad V^3=\Span{e^1,e^4,e^2-\mu e^5},\\ V^4=\Span{e^1,e^4,e^2-\mu e^5,e^3+e^6}.\qedhere
\end{gather*}
\end{proof}

\section{$7$-dimensional qc manifolds with non-vanishing torsion endomorphism
and closed fundamental $4$-form}

The purpose of this section is to prove Theorem \ref{main-theorem}.
For this, we consider the simply connected solvable Lie group 
$G_s$ $(s=1,2)$
of dimension $7$ whose Lie algebra 
$\lie{g}_s$
is defined by \eqref{solvable-examples} 
considering there $\mu=-1$ for $\lie{g}_1$, and
$\mu=-\frac{1}{3}$ for $\lie{g}_2$. 
We show that $G_s$
has an integrable left invariant qc structure
such that the fundamental $4$-form is closed, but
the torsion endomorphism of the Biquard connection is non-zero.
Firstly, we need some definitions and results about the
torsion endomorphism of the Biquard connection on a qc manifold.

Let $M$ be a manifold of dimension $4n+3$ with
a qc structure that we suppose integrable when $n=1$. According to
Section $2$, we know that  $M$ has a
distribution $H$ of dimension $4n$, locally determined by the kernel
of three differential 1-forms $\eta_r$ on $M$, and such that there is
an almost quaternion Hermitian structure $(g, I_1, I_2,I_3)$ on $H$
satisfying $2g(I_rX,Y) = d\eta_r(X,Y)$, for $r=1,2,3$ and for any $X,Y\in H$.
Let us consider the local
vector fields $\{\xi_1,\xi_2,\xi_3\}$ on $M$ satisfying
\eqref{bi1} for $n\geq 1$ \cite{D}. Using these vector fields $\xi_r$, we extend the metric $g$ on $H$
to a metric on $M$ (that we also write with the same letter $g$) by requiring
$span\{\xi_1,\xi_2,\xi_3\}=V\perp H$  and
$g(\xi_r,\xi_k)=\delta_{rk}$.

Since the Biquard connection $\nabla$ on $M$ is metric, it is related to the
Levi-Civita connection
$\nabla^g$ of the metric $g$ on $M$ by
\begin{equation}  \label{lcbi}
g(\nabla_AB,C)=g(\nabla^g_AB,C)+\frac12\Big[
g(T(A,B),C)-g(T(B,C),A)+g(T(C,A),B)\Big],
\end{equation}
where $A,B,C$ are arbitrary vector fields on $M$ and $T$ is the torsion of $\nabla$.

Let $R=[\nabla,\nabla]-\nabla_{[\ ,\ ]}$ be the curvature tensor of
$\nabla$. We denote the curvature tensor
of type (0,4) by the same letter, $R(A,B,C,D)=g(R(A,B)C,D)$, for any vector fields
$A,B,C,D$ on $M$. The \emph{qc-Ricci
$2$-forms $\rho_r$ $(r=1,2,3)$ and the \emph{normalized qc-scalar curvature} $S$} of the Biquard connection
are
defined by
\begin{equation}  \label{qscs}
4n\rho_{r}(A,B)=R(A,B,e_a,I_re_a), \quad 8n(n+2)S=R(e_b,e_a,e_a,e_b),
\end{equation}
where $\{e_1,\dots,e_{4n}\}$
is a local orthonormal basis of the distribution $H$.

Regarding the torsion
endomorphism $T_{\xi}=T(\xi,\cdot) : H\rightarrow H$, $ \xi\in V$,
 Biquard shows in
\cite{Biq1} that it is completely trace-free, i.e.
$tr\, T_{\xi}=tr\, T_{\xi}\circ
I_r=0$, and for $7$-dimensional qc manifolds the skew-symmetric part
of
$T_{\xi}: H\rightarrow H$ vanishes, so
$T_{\xi}$ only has symmetric part.

Now,
we consider the 2-tensor $T^0$ on $H$ defined by
$$
T^0(X,Y)=
g((T_{\xi_1}I_1+T_{\xi_2}I_2+T_{ \xi_3}I_3)X,Y),
$$
for $X,Y\in H$. In \cite{IMV} it is proved
\begin{equation}  \label{propt}
\begin{aligned}
T^0(X,Y)+T^0(I_1X,I_1Y)+T^0(I_2X,I_2Y)+T^0(I_3X,I_3Y)=0. \\
\end{aligned}
\end{equation}
Moreover, taking into account \cite[Proposition~2.3]{IV},
on a  {\em seven dimensional}
qc manifold, the torsion endomorphism satisfies the following relations
\begin{equation}  \label{need}
4g(T_{\xi_r}(I_rX),Y)=T^0(X,Y)-T^0(I_rX,I_rY), \quad r=1,2,3.
\end{equation}
In order to determine the torsion endomorphism of the Biquard connection on M, we need
know the differential 1-forms $\alpha_r$
such that
$$
\nabla I_i=-\alpha_j\otimes I_k+\alpha_k\otimes I_j,\quad
\nabla\xi_i=-\alpha_j\otimes\xi_k+\alpha_k\otimes\xi_j,
$$
where from now on $(i,j,k)$ is an arbitrary cyclic permutation of  $(1,2,3)$. 
The $1$-forms $\alpha_r$ are called the $sp(1)$-connection forms. Biquard
in \cite{Biq1} shows that
on $H$ they are expressed
by
\begin{gather}  \label{coneforms}
\alpha_i(X)=d\eta_k(\xi_j,X)=-d\eta_j(\xi_k,X), \quad X\in H, \quad
\xi_i\in V,
\end{gather}
while
on the distribution $V$
they are given by (see \cite{IMV})
\begin{gather}  \label{coneform1}
\alpha_i(\xi_s)=d\eta_s(\xi_j,\xi_k)-\
\delta_{is}\left(\frac{S}2\ +\ \frac12\,\left(\,
d\eta_1(\xi_2,\xi_3)\ +\ d\eta_2(\xi_3,\xi_1)\ +\
d\eta_3(\xi_1,\xi_2)\right)\right),
\end{gather}
where  $S$ is the \emph{normalized} qc scalar curvature defined
by \eqref{qscs}. We notice that in  \cite{IMV} it is proved that
the vanishing of the $sp(1)$-connection
$1$-forms on $H$ implies the vanishing of the torsion endomorphism of
the Biquard connection.

The qc Ricci 2-forms are determined by the
$sp(1)$-connection 1-forms $\alpha_r$ as follows
\begin{equation}  \label{sp1curv}
2\rho_k(A,B)=(d\alpha_k+\alpha_i\wedge\alpha_j)(A,B),
\end{equation}
for any vector fields $A,B$ on $M$.
Moreover (see below \eqref{sixtyfour}), the qc Ricci forms
restricted to $H$ can be expressed in terms of the endomorphism
torsion of the Biquard connection.
We collect the necessary facts from
\cite[Theorem~4.3.5]{IV2}
for $7$-dimensional qc manifolds, so the torsion
endomorphism $T_{\xi}$ only has symmetric part.

\begin{theorem}\label{sixtyseven}
\cite{IMV} On a $7$-dimensional qc manifold $(M,\eta,\mathbb{Q})$ the following
formulas hold:
\begin{equation} \label{sixtyfour}
\begin{aligned}
\rho_{r}(X,Y) \ & =\
\frac12\Bigl[T^0(X,I_{r}Y)-T^0(I_{r}X,Y)\Bigr]-S\omega_{r}(X,Y),\\
T(\xi_{i},\xi_{j})& =-S\xi_{k}-[\xi_{i},\xi_{j}]_{H}, \qquad S\  =\
-g(T(\xi_1,\xi_2),\xi_3),\\
\end{aligned}
\end{equation}
where $r=1,2,3$ and $X,Y\in H$.
\end{theorem}

\subsection{Example 1 $(\mu=-1)$
}
Consider the simply connected solvable (non-nilpotent)
Lie group $G_1$ of dimension $7$ whose Lie algebra  is
defined by the equations
\begin{equation}  \label{ex1}
\begin{aligned}
 de^1=&0, \\
 de^2=&(1/2)e^{15}-e^{3 4}+(1/2)e^{4 6}, \\
 de^3=&(1/2)e^{16}+e^{24}-(1/2)e^{45}, \\
 de^4=&-2 e^{1 4}, \\
 de^5=&2 (e^{1 2}+e^{3 4})-e^{4 6}, \\
 de^6=&2 (e^{1 3}+ e^{42})+e^{4 5}, \\
 de^7=&2 (e^{1 4}+e^{2 3})-(1/2)e^{5 6}.
\end{aligned}
\end{equation}
We must notice that this Lie algebra is isomorphic to the Lie
algebra $\mathfrak {g_1}$
defined by \eqref{solvable-examples}
for $\mu=-1$. In fact, considering
the basis
$\{f^j; 1\leq j\leq 7\}$ of  $\mathfrak {g_1}^*$
given by $f^j=e^j$ for $1\leq j\leq 4$, and $f^j=2 e^j$ for $5\leq j\leq 7$,
equations \eqref{solvable-examples} with $\mu=-1$ become \eqref{ex1},
where we write $e^j$ instead of $f^j$.

Let $\{e_j; 1\leq j\leq 7\}$ be the basis of left invariant vector fields on $G_1$ dual
to $\{e^j, 1 \leq j\leq 7\}$. We define a
global qc structure on the Lie group $G_1$ by
\begin{equation}  \label{qc1}
\begin{aligned}
&\eta_1=e^5, \quad \eta_2=e^6, \quad \eta_3=e^7, \quad
\xi_1=e_5, \quad \xi_2=e_6,\quad \xi_3=e_7,\\ &\mathbb
H=Span\{e^1,\dots, e^4\}, \\ &\mathbb
 \omega_1=e^{12}+e^{34}, \quad
\omega_2=e^{13}+e^{42}, \quad \omega_3=e^{14}+e^{23}.
\end{aligned}
\end{equation}
It follows from \eqref{ex1} that the triplet $\{\xi_1=e_5,
\xi_2=e_6, \xi_3=e_7\}$ defined by \eqref{qc1} are vector fields on $G_1$ satisfying
\eqref{bi1}. Therefore,
the qc structure on $G_1$
is integrable, and so the Biquard connection exists.

\begin{theorem}\label{m1}
The left invariant qc structure defined by \eqref{qc1} on $G_1$ is such that
the torsion endomorphism of the Biquard connection is
non-zero, the fundamental $4$-form is closed and the normalized qc scalar
curvature is $S=-\frac{1}{2}$.
\end{theorem}
\begin{proof}
Clearly, \eqref{ex1} and \eqref{qc1} imply that the
fundamental $4$-form $\Omega$ on $G_1$
defined by \eqref{fform} is such that $d\Omega=2 e^{1234}=0$.
The closedness of $\Omega$ can also be seen as a consequence of the 
fact that the vertical distribution is integrable. Indeed,  a result 
in~\cite[Theorem~4.7]{AFISV} states
that {\em for a qc structure in dimension $7$, the fundamental four form 
is closed if and only if the vertical distribution is integrable}.

We determine the connection $1$-forms $\alpha_{r}$ of the Biquard connection on $G_1$.
The  structure equations
\eqref{ex1} together with \eqref{coneforms} and \eqref{coneform1}
imply
\begin{equation}  \label{ex1conf}
\alpha_{1}=-\frac{1}{2} (S-\frac{1}{2}) e^5,\quad\alpha_{2}=-\frac{1}{2} (S-\frac{1}{2}) e^6,
\quad\alpha_{3}=-e^4-\frac{1}{2} (S+\frac{1}{2})e^7.
\end{equation}
Now, \eqref{sp1curv}, \eqref{ex1} and \eqref{ex1conf} yield
\begin{equation}  \label{rtor1}
\begin{aligned}
\rho_1(X,Y)&= -\frac{1}{2}(S-\frac{1}{2}) \omega_1(X,Y),\\
\rho_2(X,Y)&= -\frac{1}{2}(S-\frac{1}{2}) \omega_2(X,Y), \\
\rho_3(X,Y)&=e^{14}(X,Y)-\frac{1}{2}(S+\frac{1}{2}) \omega_3(X,Y) \\
&=\frac{1}{2}(e^{14}-e^{23})(X,Y)-\frac{1}{2}(S-\frac{1}{2}) \omega_3(X,Y),
\end{aligned}
\end{equation}
for $X,Y\in H$. Comparing \eqref{rtor1} with \eqref{sixtyfour} we conclude
$$
\begin{aligned} T^0(X,I_1Y)-T^0(I_1X,Y)=0, \qquad S=-\frac{1}{2},\\
\quad T^0(X,I_2Y)-T^0(I_2X,Y)=0,\\
\quad T^0(X,I_3Y)-T^0(I_3X,Y)= (e^{14}-e^{23})(X,Y),
\end{aligned}
$$
or, equivalently,
\begin{equation}  \label{tr1-1}
\begin{aligned}
T^0(I_1X,I_1Y)+T^0(X,Y)=0, \qquad S=-\frac{1}{2},\\
\quad T^0(I_2X,I_2Y)+T^0(X,Y)=0,\\
\quad T^0(I_3X,I_3Y)+T^0(X,Y)= - (e^{14}-e^{23})(X,I_3Y).
\end{aligned}
\end{equation}
From equations \eqref{tr1-1} and \eqref{need} we have $T_{\xi_3}=0$ and
\begin{equation}  \label{tor-endom-1}
T^0(X,Y)=-\frac{1}{2}(e^{14}-e^{23})(X,I_3Y),\\
\quad
g(T(\xi_{r},X),Y)=\frac{1}{4}(e^{14}-e^{23})(I_{r}X,I_3Y),
\end{equation}
for $r=1,2$. Equations \eqref{tor-endom-1} imply that the endomorphism torsion is non-zero.
In fact,
we have $T(e_5,e_1)=T_{\xi_1}(e_1)=-\frac{1}{4}e_2\not=0$
which completes the proof.
\end{proof}

Now, following \cite{IV},
we consider the {\em qc conformal curvature tensor} $W^{qc}$ of a
seven dimensional qc manifold with distribution $H$, that is,
the tensor on $H$ of type $(0,4)$ given by
\begin{multline}  \label{qcwdef1}
W^{qc}(X,Y,Z,V) =
R(X,Y,Z,V)+ (g\owedge L_0))(X,Y,Z,V)+
\sum_{s=1}^3(\omega_s
\owedge I_sL_0)(X,Y,Z,V) \\
-\frac12\sum_{s=1}^3\Bigl[\omega_s(X,Y)\Bigl\{T^0(Z,I_sV)-T^0(I_sZ,V)\Bigr\}
+ \omega_s(Z,V)\Bigl\{T^0(X,I_sY)-T^0(I_sX,Y)
\Bigr\}\Bigr] \\
+\frac{S}4\Big[(g\owedge g)(X,Y,Z,V)+\sum_{s=1}^3\Bigl((\omega_s\owedge
\omega_s)(X,Y,Z,V) +4\omega_s(X,Y)\omega_s(Z,V)\Bigr) \Big],
\end{multline}
where $X,Y,Z,V \in H$, $L_0=\frac{1}{2}T^0$, $I_s L_0\, (X,Y) = -L_0 (X,I_s Y)$ $(s=1,2,3)$,
and $\owedge$ is the
Kulkarni-Nomizu product of $2$-tensors, which is defined as follows. If $\mu$ and $\nu$
are $2$-tensors on $H$, then $\mu\owedge \nu$ is the $4$-tensor given by
\begin{multline} \label{Kulkarni-Nomizu}
(\mu\owedge \nu)(X,Y,Z,V):=\mu(X,Z)\nu(Y,V)+
\mu(Y,V)\nu(X,Z)-\mu(Y,Z)\nu(X,V)\\
-\mu(X,V)\nu(Y,Z),
\end{multline}
for any $X,Y,Z,V\in H$.

The tensor $W^{qc}$ is the obstruction for a qc structure to be locally qc
conformal to the flat structure on the quaternionic Heisenberg group.

\begin{theorem} \label{main1} \cite[Theorem~4.4]{IV} A qc structure on a
smooth manifold of dimension $4n+3$ is locally qc conformal to the standard flat qc
structure on the quaternionic Heisenberg group
$G(\mathbb{H})$ if and only if $W^{qc}=0$. In this case,
the qc structure is said to be a qc conformally flat structure.
\end{theorem}

\begin{proposition}
The left invariant qc structure defined by \eqref{qc1} on $G_1$
is not locally qc conformally flat.
\end{proposition}
\begin{proof}
Using \eqref{tor-endom-1}  and \eqref{Kulkarni-Nomizu} we 
see that the tensor $W^{qc}$ given by \eqref{qcwdef1} satisfies
\begin{multline} \label{wqc-1}
W^{qc}(e_1,e_2,e_1,e_2)=R(e_1,e_2,e_1,e_2)+\frac{S}4\Big[(g\owedge g)(e_1,e_2,e_1,e_2)\\
+\sum_{r=1}^3\Bigl((\omega_r\owedge
\omega_r)(e_1,e_2,e_1,e_2) +4\omega_r(e_1,e_2)\omega_r(e_1,e_2)\Bigr) \Big],
\end{multline}
since
other terms on the right hand
side of \eqref{qcwdef1} vanish on the
quadruplet $(e_1,e_2,e_1,e_2)$.
It is straightforward to check from \eqref{lcbi}, \eqref{ex1}, \eqref{qc1}, 
\eqref{tor-endom-1} and \eqref{Kulkarni-Nomizu} that
$$
\begin{aligned}
 R(e_1,e_2,e_1,e_2)=\frac{1}{2}, \quad (g\owedge g)(e_1,e_2,e_1,e_2)=2, \qquad
 \\
 \sum_{s=1}^{3}((\omega_{s}\owedge \omega_{s})(e_1,e_2,e_1,e_2)+4\omega_{s}(e_1,e_2)\omega_{s}(e_1,e_2))=6.
\end{aligned}
$$
\noindent Substituting these equalities in \eqref{wqc-1} we obtain
$W^{qc}(e_1,e_2,e_1,e_2)=-\frac{1}{2}\not=0$, which completes the proof
according to Theorem \ref{main1}.
\end{proof}
\subsection{Example 2  $(\mu=-\frac{1}{3})$
}
Next, we consider the simply connected solvable (non-nilpotent) Lie group $G_2$
of dimension $7$ whose Lie algebra is
defined by
\begin{equation}  \label{ex2}
\begin{aligned}
 de^1=&0, \\
 de^2=&\frac{2}{3}e^{12}+\frac{1}{6}e^{15}-\frac{1}{3}e^{34}+\frac{1}{6}e^{46}, \\
 de^3=&-\frac{2}{3}e^{13}+\frac{1}{6}e^{16} -e^{24}-\frac{1}{6}e^{45}, \\
 de^4=&-\frac{2}{3}e^{14}, \\
 de^5=&2 (e^{1 2}+e^{3 4})-e^{46}, \\
 de^6=&2(e^{1 3}+e^{4 2})+e^{45}, \\
de^{7}=& 2(e^{14}+e^{23})-\frac{1}{6}e^{56}.
\end{aligned}
\end{equation}
This Lie algebra  is isomorphic to the Lie
algebra $\mathfrak {g_2}$
obtained in Proposition $3.1$ for $\mu=-1/3$. Indeed, taking the basis
$\{f^j; 1\leq j\leq 7\}$ of  $\mathfrak {g_2}^*$
defined by $f^j=e^j$ for $1\leq j\leq 4$, and $f^j=2 e^j$ for $5\leq j\leq 7$,
equations \eqref{solvable-examples} with $\mu=-1/3$ become \eqref{ex2},
where we write $e^j$ instead of $f^j$.

Let $\{e_j; 1\leq j\leq 7\}$ be the basis of left invariant vector fields on $G_2$ dual
to $\{e^j, 1 \leq j\leq 7\}$. We define a
global qc structure on the Lie group $G_2$ by
\begin{equation}  \label{qc2}
\begin{aligned}
&\eta_1=e^5, \quad \eta_2=e^6, \quad \eta_3=e^7, \quad
\xi_1=e_5,\quad \xi_2=e_6,\quad \xi_3=e_7,\\ &\mathbb
H=Span\{e^1,\dots, e^4\}, \\ &\mathbb
\omega_1=e^{12}+e^{34}, \quad
\omega_2=e^{13}+e^{42} \quad \omega_3=e^{14}+e^{23}.
\end{aligned}
\end{equation}
From \eqref{ex2} we have that the triplet $\{\xi_1=e_5,
\xi_2=e_6, \xi_3=e_7\}$ of vector fields on $G_2$ defined by \eqref{qc2}
satisfy \eqref{bi1}. Therefore the Biquard connection does exist.

\begin{theorem}\label{m2}
The left invariant qc
structure
defined by \eqref{qc2}
on the simply connected solvable Lie
group $G_2$ is such that
the  torsion endomorphism of the Biquard connection is non-zero,
the fundamental $4$-form is closed and the normalized qc scalar
curvature is $S=-\frac{1}{6}$.
\end{theorem}
\begin{proof}
Using \eqref{ex2}, \eqref{qc2} and 
Theorem~4.7 in ~\cite{AFISV} we get that the
fundamental $4$-form $\Omega$ on $G_2$,
defined by \eqref{fform}, is closed since the vertical distribution
of the qc structure
is integrable.

On the other hand, the  structure equations
\eqref{ex2} together with \eqref{coneforms} and \eqref{coneform1}
imply
\begin{equation}  \label{ex2conf}
\alpha_{1}=-\frac{1}{2} (S-\frac{1}{6}) e^5,\quad\alpha_{2}=-\frac{1}{2} (S-\frac{1}{6}) e^6,
\quad\alpha_{3}=-e^4-\frac{1}{2} (S+\frac{1}{6})e^7.
\end{equation}
Now, from \eqref{sp1curv}, \eqref{ex2} and \eqref{ex2conf}, we get
\begin{equation}  \label{rtor2}
\begin{aligned}
\rho_1(X,Y)= -\frac{1}{2}(S-\frac{1}{6}) \omega_1(X,Y),\\
\rho_2(X,Y)= -\frac{1}{2}(S-\frac{1}{6}) \omega_2(X,Y),\\
\rho_3(X,Y)=-\frac{1}{2}(S-\frac{1}{6}) \omega_3(X,Y)+\frac{1}{6} (e^{14}-e^{23})(X,Y),
\end{aligned}
\end{equation}
for $X,Y\in H$. Comparing \eqref{rtor2} with \eqref{sixtyfour} we get
$$
\begin{aligned} T^0(X,I_1Y)-T^0(I_1X,Y)=0, \qquad S=-\frac{1}{6},\\
\quad T^0(X,I_2Y)-T^0(I_2X,Y)=0,\\
\quad T^0(X,I_3Y)-T^0(I_3X,Y)= \frac{1}{3}(e^{14}-e^{23})(X,Y),
\end{aligned}
$$
or, equivalently,
\begin{equation}  \label{tr1-2}
\begin{aligned}
T^0(I_1X,I_1Y)+T^0(X,Y)=0, \\
\quad T^0(I_2X,I_2Y)+T^0(X,Y)=0,\\
\quad T^0(I_3X,I_3Y)+T^0(X,Y)= -\frac{1}{3} (e^{14}-e^{23})(X,I_3Y),
\end{aligned}
\end{equation}
for $X,Y\in H$. From equations \eqref{tr1-2} and taking into account \eqref{propt} and \eqref{need},
we have $T_{\xi_3}=0$ and
\begin{equation}  \label{tor-endom-2}
T^0(X,Y)=-\frac{1}{6}(e^{14}-e^{23})(X,I_3Y), \\
\quad
g(T(\xi_r,X),Y)=\frac{1}{12}(e^{14}-e^{23})(I_rX,I_3Y),
\end{equation}
for $r=1,2$. Equations \eqref{tor-endom-2} imply that the endomorphism torsion is non-zero.
In fact, $T(e_5,e_1)=T_{\xi_1}(e_1)=-\frac{1}{12}e_2\not=0$.
 \end{proof}

\begin{proposition}
The left invariant qc structure defined by \eqref{qc2} on $G_2$
is not locally qc conformally flat.
\end{proposition}
\begin{proof}
Using \eqref{qcwdef1}, \eqref{Kulkarni-Nomizu} and \eqref{tor-endom-2} we have that
the expression of $W^{qc}(e_1,e_2,e_1,e_2)$ becomes as \eqref{wqc-1}.
From \eqref{lcbi}, \eqref{Kulkarni-Nomizu}, \eqref{ex2}, \eqref{qc2} and
\eqref{tor-endom-2} we obtain
$$
\begin{aligned}
 R(e_1,e_2,e_1,e_2)=\frac{11}{18}, \quad (g\owedge g)(e_1,e_2,e_1,e_2)=2, \qquad
 \\
 \sum_{r=1}^{3}((\omega_{r}\owedge \omega_{r})(e_1,e_2,e_1,e_2)+4\omega_{r}(e_1,e_2)\omega_{r}(e_1,e_2))=6.
 \end{aligned}
$$
\noindent Therefore,
$W^{qc}(e_1,e_2,e_1,e_2)=-\frac{5}{18}\not=0$. The result follows from Theorem \ref{main1}.
\end{proof}

\small\noindent Dipartimento di Matematica e Applicazioni, Universit\`a di Milano Bicocca,  Via Cozzi 53, 20125 Milano, Italy.\\
\texttt{diego.conti@unimib.it}

\smallskip
\small\noindent Universidad del Pa\'{\i}s Vasco, Facultad de Ciencia y Tecnolog\'{\i}a, Departamento de Matem\'aticas,
Apartado 644, 48080 Bilbao, Spain. \\
\texttt{marisa.fernandez@ehu.es}\\
\texttt{joseba.santisteban@ehu.es}


\begin{thebibliography}{AHit}


\bibitem{Alek} D.~V.~Alekseevski{\u\i}, B.~N.~Kimel$'$fel$'$d,
Structure of homogeneous Riemannian spaces with zero Ricci curvature. (Russian)
Funkcional. Anal. i Prilo\v Zen. {\bf 9}(2) (1975), 5-11 (English translation: Functional Anal. Appl. 9 (1975), 97-102).

\bibitem{Alt}
J.~Alt,
On the twistor space of a quaternionic contact manifold,
{\em J. Geom. Phys.\/} {\bf 61} (2011), 1783-1788.

\bibitem{Biq1}
O.~Biquard,
Quaternionic contact structures,
{\em Quaternionic structures in mathematics and physics} (Rome, 1999),
23--30 (electronic), Univ. Studi Roma "La Sapienza", Roma, 1999.

\bibitem{Biq2}
O.~Biquard,
{\em M\'{e}triques d'Einstein asymptotiquement sym\'{e}triques},
Ast\'{e}risque {\bf 265} (2000).

\bibitem{BGN}
Ch.~Boyer, K.~Galicki, B.~Mann,
The geometry and topology of $3$-Sasakian manifolds,
{\em J. Reine Angew. Math.\/} {\bf 455} (1994), 183--220.

\bibitem{AFISV}
L.~C. de Andr\'es, M.~Fern\'andez, S.~Ivanov, J.A.~Santisteban, L.~Ugarte,
D.~Vassilev,
Quaternionic K\"ahler and Spin(7) metrics arising from quaternionic
contact Einstein structures.
Preprint 2010,  arXiv:1009.2745.

\bibitem{DIM}
J.~Davidov, S.~Ivanov, I.~Minchev,
The twistor space of a quaternionic contact manifold.
To appear in {\em Quart. J. Math.\/} Preprint 2010,  arXiv:1010.4994.

\bibitem{D}
D.~Duchemin,
Quaternionic contact structures in dimension 7,
{\em Ann. Inst. Fourier} {\bf 56} (4) (2006), 851--885.

\bibitem{D1}
D.~Duchemin,
Quaternionic contact hypersurfaces.
Preprint 2010,  arXiv:math/0604147.


\bibitem{IMV1}
S.~Ivanov, I.~Minchev, D.~Vassilev,
Extremals for the Sobolev inequality on the seven dimensional quaternionic
Heisenberg group and the quaternionic contact Yamabe problem,
{\em J. Eur. Math. Soc.\/} {\bf 12} (4) (2010), 1041--1067.

\bibitem{IMV}
S.~Ivanov, I.~Minchev, D.~Vassilev,
Quaternionic contact Einstein structures and the quaternionic
contact Yamabe problem. Preprint 2010, arXiv:math/0611658.


\bibitem{IV}
S.~Ivanov, D.~Vassilev,
Conformal quaternionic contact curvature and the local sphere theorem,
{\em J. Math. Pures Appl.\/} {\bf 93} (2010), 277--307.

\bibitem{IV1}
S.~Ivanov, D.~Vassilev,
Quaternionic contact manifolds with a closed fundamental 4-form,
{\em Bull. London Math. Soc.\/} {\bf 42} (2010), 1021--1030.

\bibitem{IV2}
S.~Ivanov, D.~Vassilev,
{\em Extremals for the Sovolev inequality and the
quaternionic contact Yamabe problem}, World Scientific Publish., 2011.


\bibitem{Wei}
W.~Wang,
The Yamabe problem on quaternionic contact manifolds,
{\em Ann. Mat. Pura Appl.\/} {\bf 186} (2) (2007), 359--380.

\end{thebibliography}
\end{document}